\documentclass[10pt,reqno]{amsart} 
\usepackage{amsthm,amsmath,amsfonts,amssymb, cases}
\usepackage{hyperref}
\usepackage[a4paper, left=2.5cm, right=2.5cm, top=3cm, bottom=3cm]{geometry}
\usepackage[dvipsnames]{xcolor}
\usepackage{tikz}
\usetikzlibrary{shapes.geometric}
\usetikzlibrary{arrows.meta}
\usepackage[capitalise]{cleveref}
\usepackage[T1]{fontenc}
\usepackage[utf8]{inputenc}
\hypersetup{linktoc=section, colorlinks=true, linkcolor=Red, citecolor=Green, urlcolor=Blue}

\numberwithin{equation}{section}
\newtheorem{Theorem}{Theorem}[section]
\newtheorem{Lemma}[Theorem]{Lemma}
\newtheorem{Proposition}[Theorem]{Proposition}
\newtheorem{Corollary}[Theorem]{Corollary}

\newtheorem{Remark}[Theorem]{Remark}

\renewcommand{\leq}{\leqslant}

\renewcommand{\geq}{\geqslant}

\DeclareMathOperator{\PWC}{PWC}
\DeclareMathOperator{\supp}{supp}

\DeclareMathOperator{\Span}{span}

\DeclareMathOperator{\Card}{Card}

\DeclareMathOperator{\Tr}{tr}

\DeclareMathOperator{\Mat}{Mat}

\def\C{{\mathbb{C}}}

\def\N{{\mathbb{N}}}

\def\R{{\mathbb{R}}}
\def\S{{\mathbb{S}}}
\def\T{{\mathbb{T}}}

\def\Z{{\mathbb{Z}}}

\usepackage{enumitem}

\setlist[itemize]{
  leftmargin=15pt,   
  rightmargin=15pt,
  itemsep=0pt,       
  parsep=0pt,       
  topsep=0pt,       
  partopsep=0pt
}
\setlist[itemize,1]{label=\textbullet}
\setlist[itemize,2]{label=--}
\setlist[itemize,3]{label=$\triangleright$}
\setlist[itemize,4]{label=$\cdot$}

 \setlist[enumerate]{
  leftmargin=15pt,   rightmargin=15pt, itemsep=0pt, parsep=0pt, topsep=0pt, partopsep=0pt
}
\setlist[enumerate,1]{label=(\roman*), ref=(\roman*)}
\setlist[enumerate,2]{label=(\alph*),  ref=\theenumi(\alph*)}
\setlist[enumerate,3]{label=\arabic*.}

\title[Saturating spaces on compact Riemannian manifolds]{Saturating spaces for the squared-gradient nonlinearity on compact Riemannian manifolds via Nash embeddings
}

\author[R. Buffe]{Rémi Buffe}
\address{Université de Lorraine, CNRS, Inria, IECL, F-54000 Nancy, France.}
\email{remi.buffe@univ-lorraine.fr}

\author[A. Duca]{Alessandro Duca}
\address{Université de Lorraine, CNRS, Inria, IECL, F-54000 Nancy, France.}
\email{alessandro.duca@inria.fr}

\author[H. Parada]{Hugo Parada}
\address{Université de Lorraine, CNRS, Inria, IECL, F-54000 Nancy, France.}
\email{hugo.parada@inria.fr}
\thanks{This project is funded by the Agence Nationale de la Recherche by the QuBiCCS project (ANR-24-CE40-3008)}

\begin{document}
\begin{abstract}
 We study the saturation property on smooth compact Riemannian manifolds without boundary. Starting from a finite-dimensional space of smooth functions, one considers successive iterations of the nonlinear map $f\mapsto|\nabla_g f|_g^2$ and convex combinations. We extend to the general setting of compact manifolds existing results on particular geometries, namely the torus, spheres and Lie groups. Our approach relies on the existence of an isometric embedding of the manifold into $\mathbb{R}^n$ due to Nash \cite{Nash}. As an application, we obtain global approximate controllability in arbitrarily small time for the bilinear Schr\"odinger equation on any connected compact Riemannian manifold, with finitely many potentials.
\end{abstract}
\maketitle
\section{Introduction and main results}
A classical problem in applied mathematics and physics is to drive a dynamical system from one state to another by acting on it through a limited number of external inputs, the controls. For partial differential equations,
the most studied setting is that of additive controls, where the input enters the equation as a source term. Additive controls, however, can only inject or
remove matter, energy or probability, and they are not suited to situations where one acts on the system by modifying its parameters, for instance the potential felt by a quantum particle exposed to an electric field, the
reaction rate in a chemical process, or the coefficients of a diffusion. In such problems, the control should act multiplicatively on the state and in the typical
case where the spatial profile of the action is fixed and only its intensity can be controlled in time, one adds to the equation one or more terms of the form
$$
  u(t)\,\mu(x)\,\psi ,
$$
where $\psi$ is the unknown of the evolution, $u$ is the control, and the fixed function $\mu$ describes the spatial profile of the action. The resulting system is called bilinear since it is linear in $\psi$ and in $u$.

\smallskip

\noindent
Global approximate controllability in arbitrarily small time with bilinear controls has been established only recently. It was obtained for the Schr\"odinger equation in \cite{DN24,BP25}, for the wave equation in \cite{Poz24,BPP26}, and for parabolic equations in \cite{DPU25,DucaTakahashi2025} (see also \cite{BCP26,BCS13,BDP26KdV,P26} for further developments). All these results build on the work \cite{DN24} which introduced a multiplicative version of the saturation method developed by
Agrachev and Sarychev for additive control problems \cite{AS05, AS06}, further significantly extended by \cite{BP25},. At the core of the multiplicative saturation method lies a density property. Heuristically speaking, starting from the vector space spanned by the control potentials $\mu_1,\dots,\mu_q$, one has to determine whether the set obtained by iterating convex combinations and the nonlinear operation 
\begin{equation}\label{11}\mu\mapsto |\nabla_g \mu|_g^2\end{equation} 
is dense in a suitable function space. The existing results have been established only in specific settings, relying on the algebraic structure of the potentials: trigonometric functions on the torus \cite{DN24} and on the sphere \cite{CP23}, Hermite functions on $\R^d$ \cite{BP25} (the case of an interval has been treated in \cite{DucaTakahashi2025}). All of these articles consider very specific functions as potentials, and
exploit their algebraic properties to ensure that iterating the saturation process generates a dense space.  In related work \cite{Par26}, saturating families are found in Lie groups and normal homogeneous spaces (quotient of Lie groups by closed subgroups). We recall that, with respect to control results, the density property of the map \eqref{11} guarantees small-time controllability of phases for general smooth compact Riemannian manifolds \cite[Theorem~6]{CP23} and \cite[Theorem~A.3]{BLS25}.

\smallskip

\noindent 
The aim of this work is to establish the saturation property for the map \eqref{11} on a general smooth compact Riemannian manifold $M$. In this setting, the algebraic structures of trigonometric functions and of Hermite functions, which are exploited in previous works, are no longer available. Our approach relies on the existence of an isometric embedding $\Theta : M \to \mathbb{R}^n$, provided by the Nash embedding theorem. As potentials, we take the functions $\cos(k\cdot\theta)$ and $\sin(k\cdot\theta)$, with $k\in\mathbb{Z}^n$ and $|k|_\infty\leqslant 1$, where $\theta$ is a suitable rescaling of $\Theta$. The key observation is that, since $\Theta$ is isometric, the quantity $\sum_{i=1}^n |\nabla_g \theta_i|_g^2$ is constant on $M$. This identity allows us to carry out a combinatorial argument on $\mathbb{Z}^n$, which shows that all the functions $\cos(k\cdot\theta)$ and $\sin(k\cdot\theta)$, $k\in\mathbb{Z}^n$, are generated. Density then follows from a Stone--Weierstrass-type theorem in $C^\infty(M)$ due to Nachbin \cite{Nachbin1949}. Beyond the application to the Schrödinger equation (see Theorem~\ref{th:bilinear_control} below), we expect our construction to provide geometric tools for extending results on approximate controllability, unbounded growth of Sobolev norms, and topological irreducibility of randomly forced PDEs from very specific geometries to more general compact Riemannian manifolds. The square-gradient structure also suggests applications to control problems for KPZ-type equations. Such extensions require analytical estimates adapted to each case. The present article can be viewed as a bilinear counterpart to the additive saturation results \cite{BDP26}, where the authors obtained a necessary and sufficient condition for iterations of polynomial maps of the form
\[
\mu\mapsto \mu^p,\quad p\in\N,\;p\geq2.
\]
In that setting, general embeddings coming from the Whitney theorem are used. Here, due to the presence of the Riemannian gradient in \eqref{11}, it is natural to additionally require the embedding to be isometric.

\subsection{Definitions of saturation sets and main result}
Let $(M,g)$ be a smooth compact Riemannian manifold of dimension $d\geq 1$ without boundary. For $f,h\in C^\infty(M)$, we define the following nonlinear maps
\[
\mathsf Q(f) := |\nabla_g f|^2_g,\quad \mathsf B(f,h) = g(\nabla_g f,\nabla_g h). 
\]
Let $n\in\N^*$ and let $\theta:M\longrightarrow \R^n$ be a smooth map. Let $\mathsf E$ be the finite dimensional subspace of $C^\infty(M;\R)$ given by 
\[
\mathsf E := \Span_\R\{ \cos( k\cdot \theta ),\sin( k\cdot \theta ), \; k\in\Z^n,\; |k|_\infty\leq 1\}.
\]
We define the following increasing sequence of sets 
\begin{equation}\label{eq:satu_space_def}
\mathsf H_0 = \mathsf E,\quad \mathsf H_{j+1} = \Big\{ h_0+\sum_{k=1}^{k_0}\mathsf Q(h_k),\;k_0\in\N,\;h_0,\dots,h_{k_0}\in\mathsf H_j\Big\},\quad \mathsf H_\infty = \bigcup_{j\geq 0}\mathsf H_j.
\end{equation}
Our main result is as follows.
\begin{Theorem}\label{thm:main}
   Let $(M,g)$ be a smooth compact Riemannian manifold without boundary, and let $n\in\N$ and $\Theta : M \rightarrow \R^n$ be any smooth isometric embedding (it always exist by the Nash embedding Theorem \ref{thm:Nash}). With the normalization $\theta=\|\Theta\|^{-1}_{L^\infty(M)}\Theta$, the set $\mathsf H_\infty$, defined in \eqref{eq:satu_space_def}, is dense in $C^\infty(M;\R).$ 
\end{Theorem}
For the sake of completeness, we provide in Section \ref{Sec:Riemannian} a short summary on Riemannian geometry.

\begin{Remark}A few remarks are in order.
\begin{enumerate}
    \item Concerning the dimension $n\in\N$ in Theorem \ref{thm:main}, Nash \cite{Nash} proved that one can take $n= \tfrac{d(3d+11)}{2}$, where $d$ is the dimension of $M$. This bound was later improved by G\"unther \cite{Gunther}, who obtained $n = \max (\tfrac{d(d+5)}{2},\tfrac{d(d+3)}{2}+5)$.
    \item We have $\dim \mathsf E\leq 3^n$, but the dimension depends on the embedding $\Theta$.
        \end{enumerate}
\end{Remark}
In fact, we shall prove the following stronger result, which implies Theorem \ref{thm:main}. 
\begin{Theorem}\label{thm:Hcontainsallfrequencies}
       Let $(M,g)$ be a smooth compact Riemannian manifold without boundary, and let $n\in\N$ and $\Theta : M \rightarrow \R^n$ be any smooth isometric embedding (it always exists by the Nash embedding Theorem \ref{thm:Nash}). With the normalization $\theta=\|\Theta\|^{-1}_{L^\infty(M)}\Theta$,
    \[
    \Span_\R\{ \cos (k\cdot\theta),\sin(k\cdot\theta),\; k\in\Z^n\} \subset \mathsf H_\infty.
    \]
\end{Theorem}
The proofs of Theorem \ref{thm:Hcontainsallfrequencies} and, as a consequence, of Theorem \ref{thm:main} are given in Section \ref{Sec:Proof}.

\subsection{Bilinear control of the Schr\"odinger equation}
We provide in this section one of the possible consequences that Theorem \ref{thm:main} may imply. Let $T>0$ and $q\in\N$. We consider the following Schr\"odinger equation 
\begin{equation}\label{eq:LS}
\begin{cases}
i\partial_t \psi +\Delta_g \psi = (u\cdot \mu)\psi  & \text{in }(0,T)\times M \\
\psi_{|_{t=0}} = \psi_0 \in L^2(M), \\
\end{cases}
\end{equation}
where $-\Delta_g$ denotes the non-negative Laplace-Beltrami operator on $M$, $\mu=(\mu_1,...,\mu_q)\in C^\infty(M;\R^q)$ are fixed potentials, and $u\in L^2(0,T;\R^q)$ is the control. We recall that there exists a unique (mild) solution in $C^0([0,T];L^2(M))$ of \eqref{eq:LS} given by the Duhamel formula
\[
\psi(t) = e^{it\Delta_g} \psi_0 -i \int_0^te^{i(t-s)\Delta_g} (u(s)\cdot \mu)\psi(s) ds,\quad t\in[0,T].  
\]
Note that, as the controls are real valued, the $L^2$--norm is preserved over time
\[
\|\psi(t)\|_{L^2(M)} = \|\psi_0\|_{L^2(M)}.
\]
\begin{Theorem}
\label{th:bilinear_control}
Let $(M,g)$ be a smooth compact connected Riemannian manifold without boundary. There exist $q\in \N$, real valued potentials $\mu\in C^\infty(M;\R^q)$ such that for any $\delta>0$, for any $\psi_0,\psi_1\in L^2(M)$ such that $\|\psi_0\|_{L^2(M)}=\|\psi_1\|_{L^2(M)}$, there exist $T\in(0,\delta]$ and $u \in \PWC([0,T],\R^{q})$ such that the unique solution of \eqref{eq:LS} satisfies
 \[
 \|\psi(T)-\psi_1\|_{L^2(M)}<\delta.
 \]
 Moreover, we have the bound $q\leq 3^n$, where $n$ is the smallest integer such that $M$ embeds isometrically into $\R^n$.
\end{Theorem}

The proof of Theorem \ref{th:bilinear_control} is given at the end of Section \ref{Sec:Proof}.

\begin{Remark}
 Of course, as this theorem is very general, the bound on $q$ may not be optimal in particular cases. For instance \cite{DN24,BP25}, when $M=\T^d$, then $q=2d$ suffices. When $M=\S^2$ \cite{CP23}, then $q=3$ is sufficient. In the case where $M$ is a Lie group with a bi-invariant metric \cite{Par26} the number of controls is related to the number of matrix coefficients of a finite family of irreducible representations. As the value of the dimension $n\in\N$ is related to the number of controls we need in the setting of Theorem \ref{th:bilinear_control}, it is thus relevant to consider the smallest possible value of $n$.
 \end{Remark}
Theorem \ref{th:bilinear_control} is the counterpart of \cite[Theorem 3]{BP25}, only stated on the $d-$dimensional torus (as far as compact manifolds are concerned), and the related work \cite[Theorem 1.1 and Theorem 1.2]{Par26} in Lie groups and normal homogeneous spaces. Up to the author's knowledge, this is the first result on general compact Riemannian manifolds.

\begin{Remark}
As another application of Theorem \ref{th:bilinear_control}, it should be possible to generalize the growth of Sobolev norms result of \cite[Theorem 5.1]{DN24} to compact Riemannian manifolds. More precisely, let $s>0$ and $H^s(M):=D((I-\Delta_g)^{s/2})$. If $u$ is replaced by a random process $\eta$ which is piecewise given by independent and identically distributed random variables in $L^2(0,1;\R^q)$ with non-degenerate law, then for any non-zero $\psi_0\in H^s(M)$ the trajectory of \eqref{eq:LS} is expected to be almost surely unbounded in $H^s(M)$.
\end{Remark}

\section{Preliminaries on Riemannian manifolds}\label{Sec:Riemannian}
\subsection{Tangent spaces, local charts and differentials}
Let $(M,g)$ be a smooth compact Riemannian manifold of dimension $d\geq 1$ without boundary. We first recall some well-known facts on Riemannian geometry. Let $m\in M$, and consider $\mathcal C =(\varphi ,\mathcal O)$ a local chart around $m$, that is, $\mathcal O$ is an open neighborhood of $m$ in $M$ and $$\varphi : \mathcal O \longrightarrow \varphi(\mathcal O) \subset \R^d$$ is a smooth diffeomorphism.
We define $T_m M$ as the space of equivalence classes of paths $\gamma :(-\varepsilon,\varepsilon)\longrightarrow M$ with the equivalence relation $\gamma \sim \widetilde\gamma$ if and only if
$\gamma(0)=\widetilde \gamma(0) =m$ and 
\[
\frac{d}{dt} \varphi(\gamma(t))_{|_{t=0}} = \frac{d}{dt} \varphi(\widetilde\gamma(t))_{|_{t=0}}.
\]
We recall that such an equivalence class is independent of the chosen chart. We then write $v=[\gamma]\in T_mM$ and its action on $f\in C^\infty(M)$ is defined as the directional derivative
$$
v f := \frac{d}{dt} f(\gamma(t))_{|_{t=0}},
$$
which, by the chain rule, does not depend on the representative $\gamma$ and uniquely determines $v$. We can then identify $v$ with this operator.
We then define $\partial_i \in T_mM$ by its action on functions $f\in C^\infty(M)$ by
\[
\partial_i f(m) = \partial_{x_i}( f\circ \varphi^{-1})(\varphi(m)).
\]
As $(\partial_1,\dots,\partial_d)$ forms a basis of $T_mM$, we can write for any $v\in T_mM$, 
\[
v = \sum_{i=1}^d v_i \partial_i.
\]
Let $N$ be a smooth manifold of dimension $q\geq 1$. Let $\phi : M \longrightarrow N$ be a smooth function. We consider $T_m\phi$ as the linear map
\[
 \begin{array}{rcl} 
T_m \phi :T_m M & \longrightarrow &T_{\phi(m)}N \\
v=[\gamma] & \longmapsto & [\phi\circ\gamma].
\end{array}
\]
\begin{itemize}
\item In the particular case where $\phi=\varphi$, we have $T_m\varphi(\partial_i)=e_i$, and thus the isomorphism
\[
 \begin{array}{rcl} 
T_m \varphi :T_m M & \longrightarrow &\R^d \\
v=\sum_{i=1}^d v_i\partial_i & \longmapsto & (v_1 ,\dots , v_d)^T.
\end{array}
\]
\item Assume in addition that $N=\R^q$. Then we have the existence of a linear map
\[
 \begin{array}{rcl} 
\mathcal J:\R^d & \longrightarrow &\R^q \\
 (v_1,\dots,v_d)^T & \longmapsto &  T_m\phi \circ T_m\varphi^{-1} (v_1 ,\dots,v_d)^T.
\end{array}
\]
We denote by $J\in \Mat_{q,d}(\R)$ its matrix in the canonical bases of $\R^d$ and $\R^q$, that is, $J=(\partial_j\phi_i(m))_{1\leq i \leq q,1\leq j\leq d}$.
\end{itemize}

\subsection{Riemannian metrics and isometric embeddings}
Here we assume that $(M,g)$ and $(N,h)$ are two Riemannian manifolds of dimension $d\geq 1$ and $q\geq 1$ respectively. Let $\phi : M \longrightarrow N$ be a smooth map. We then define $\phi^*h$ as the pullback of the Riemannian metric $h$ by
\[
\phi^*h_m(u,v) = h_{\phi(m)}(T_m\phi u , T_m \phi v),\quad u,v \in T_mM.
\]
In the case where $\phi = \varphi^{-1}$, and for any $u,v\in T_mM$, we have
    \[
    ((\varphi^{-1})^* g)_{\varphi(m)}(T_m\varphi u,T_m\varphi v) = g_m(u,v).
    \]
   Moreover, we have the existence of a matrix $G=G_{i,j}\in C^\infty(\varphi(\mathcal O);\Mat_{d,d}(\R))$ such that 
       \[
    ( G T_m\varphi u,T_m\varphi v)_{\R^d} = g_m(u,v).
    \]
    The matrix $G$ is a representative of the metric $g$ in the local chart. Moreover, for a smooth function $f\in C^\infty(M;\R)$, we have, in the local chart, $T_m\varphi(\nabla_g f) = G^{-1}\nabla_x (f\circ\varphi^{-1})$. We recall now the well-known Nash isometric embedding theorem \cite{Nash}.
\begin{Theorem}[Nash, 1956]\label{thm:Nash}
Let $(M,g)$ be a smooth Riemannian compact manifold. There exist $n\in \N$ and a smooth embedding $\Theta:M\longrightarrow \R^n$ such that $g=\Theta^*\delta,$ where $\delta$ denotes the Euclidean metric on $\R^n.$
\end{Theorem}

In what follows, $\Theta$ shall always denote an isometric embedding. We thus define $J\in \Mat_{n,d}(\R)$ as the representative of the linear map $\mathcal J := T_m\Theta\circ T_m\varphi^{-1}$ in the canonical basis. From the chain rule, we have the identity
\begin{equation}\label{eq:gradient_of_embeddings}
     T_m\varphi(\nabla_g\Theta_i) = G^{-1} \nabla_x(\Theta_i\circ \varphi^{-1}) = G^{-1} J^{T}e_i,\quad i\in \{1,\dots,n\}.
\end{equation}
We now collect some useful consequences.

\begin{Corollary}\label{cor:G=JTJ}
With the above notations, for any $m\in M$,
\[
G = J^T J.
\]
\end{Corollary}
\begin{proof} We compute, by the relation $\Theta^* \delta=g$,
\begin{align*}
(Ge_i,e_j)_{\R^d} &= g_m( \partial_i,\partial_j) = (T_m \Theta \partial_i,T_m \Theta\partial_j)_{\R^n} \\
&=(T_m \Theta \circ T_m\varphi^{-1}e_i ,T_m \Theta\circ T_m\varphi^{-1} e_j)_{\R^n} =  ( J e_i,J e_j)_{\R^n} =( J^T J e_i,e_j)_{\R^d}. 
\end{align*}

\end{proof}

In what follows, for $\Theta$ an isometric embedding, we use the following normalized version
\begin{equation}\label{eq:normalized_embedding}
\theta = \|\Theta\|_{L^\infty(M)}^{-1} \Theta .
\end{equation}
\begin{Corollary}\label{cor:trace-constant}
The function $\sum_{i=1}^n |\nabla_g \theta_i|^2_g$ is constant on $M$. More precisely,
    \[
    \sum_{i=1}^n |\nabla_g \theta_i  |^2_g =d\; \|\Theta\|_{L^\infty(M)}^{-2}=: c^*.
    \]
\end{Corollary}
\begin{proof}
Recall that, in the local chart, we have \eqref{eq:gradient_of_embeddings}, and using Corollary \ref{cor:G=JTJ},
\[
 \sum_{i=1}^n |\nabla_g \Theta_i  |^2_g = \Tr (J G^{-1} J^T) = \Tr (J (J^TJ)^{-1} J^T) = \Tr ((J^TJ)^{-1}  J^T J)=d.
 \]
 
 The result follows by normalization.
\end{proof}
\section{The non-linear kick and saturation property}\label{Sec:Proof}
Here, $\theta$ is the normalized version of an isometric embedding defined in \eqref{eq:normalized_embedding}. We also recall that the saturation sets are defined in \eqref{eq:satu_space_def}.

\subsection{Properties of saturation sets}
In this section, we provide some basic properties on the saturation sets, in particular convex cone and vector space structures.
\begin{Lemma}\label{Lem:convex_cone}
The set $\mathsf H_\infty$ is a convex cone, stable under $\mathsf Q$.
\end{Lemma}

\begin{proof}
    We prove by induction that $\mathsf H_k$ is a convex cone for any $k\geq 0$. Initialization is immediate as $\mathsf H_0$ is a vector space. Assume that $\mathsf H_k$ is a convex cone. Let $f,h\in \mathsf H_{k+1}$ and $\lambda\geq0$. There exist $\tilde f_j,\tilde h_j \in \mathsf H_k$ such that 
    \[
    f = \tilde f_0+\sum_j \mathsf Q(\tilde f_j),\quad  h = \tilde h_0+\sum_j \mathsf Q(\tilde h_j).
    \]
    Hence, 
    \[
    f+\lambda h = \tilde f_0+ \lambda\tilde h_0 + \sum_j (\mathsf Q(\tilde f_j)+\mathsf Q(\sqrt\lambda\tilde h_j)),
    \]
    thus by the induction hypothesis $f+\lambda h\in \mathsf H_{k+1}.$ Stability under $\mathsf Q$ is immediate from the definition of $\mathsf H_\infty$.
\end{proof}
 We denote by \[\mathcal H_\infty := \mathsf H_\infty\cap (-\mathsf H_\infty),\] that is, the largest vector space contained in $\mathsf H_\infty.$ We are also led to consider 
\[
\mathcal K_\infty:=\{ f\in \mathcal H_\infty,\; -\mathsf Q(f) \in \mathsf H_\infty\}.
\]
 Exploiting the polarization formula, we have the immediate following result.
\begin{Proposition}\label{prop:stability_scalar_product}
    For any $f,h\in C^\infty(M,\R)$ such that $ \mathsf Q(f\pm h),-\mathsf Q(f), -\mathsf Q(h)\in\mathsf H_\infty$, we have $\pm\mathsf B(f,h) \in \mathsf H_\infty$. In addition, 
    $\mathcal K_\infty$ is a real vector space.
\end{Proposition}
\begin{proof}
The first property follows from the polarization identity and Lemma \ref{Lem:convex_cone}
\[
\pm \mathsf B(f,h) = \frac{1}{2}\mathsf Q( f\pm h)-\frac{1}{2}\mathsf Q(f) -\frac{1}{2}\mathsf Q(h).
\]
We now prove that $\mathcal K_\infty$ is a vector space. Let $f,h\in \mathcal K_\infty$ and $\lambda\in \R$. Hence, we have by definition
    \[
    f\pm\lambda h\in\mathcal H_\infty\subset \mathsf H_\infty,\quad \pm \mathsf Q(f),\pm  \mathsf Q(h) \in \mathsf H_\infty.
    \]
As $\mathsf H_\infty$ is stable under $\mathsf Q$, we have $\mathsf Q(f\pm\lambda h)\in\mathsf H_\infty$. Also, as $\pm \lambda^2 \mathsf Q(h) = \lambda^2 (\pm \mathsf Q(h)),$ we deduce from Lemma \ref{Lem:convex_cone} that $\pm \lambda^2\mathsf Q(h)\in \mathsf H_\infty$. In particular,
    \[
    \pm 2\lambda \mathsf B(f,h) =  \mathsf Q(f\pm\lambda h) - \mathsf Q(f) - \lambda^2\mathsf Q(h)  \in \mathsf H_\infty,
    \]
   and we deduce
        \[
    -\mathsf Q(f + \lambda h) = - \mathsf Q(f) -\lambda^2 \mathsf Q(h) - 2\lambda \mathsf B(f,h) \in \mathsf H_\infty.
    \]
    We thus have $f+\lambda h\in \mathcal K_\infty.$ 
\end{proof}

In other words, the space $\mathcal K_\infty$ is composed of reachable elements $f,h$ (that is elements of $\mathsf H_\infty$) such that $\pm \mathsf Q(\alpha f+\beta h) $ and $\mathsf B(f,h)$ are also reachable.
\subsection{
Saturation for trigonometric polynomials: initialization}
In what follows, $(e_j)_j$ denotes the canonical vectors of $\R^n$. For two vectors $\ell,m\in \Z^n$, we first introduce the notation
\[
\ell\leq m \;(\text{resp. } \ell<m) \quad\text{ if } \quad(\ell,e_j)\leq (m,e_j)\; (\text{resp. }(\ell,e_j)< (m,e_j)),
\]
for any $j\in\{1,\dots,n\}$. We also define
\[
\Lambda_{ X} := \Big\{ \ell\in \Z^n,\; \cos(\ell\cdot\theta),\sin(\ell\cdot\theta) \in  X_\infty \Big\}, \text{ where }  X = \mathcal H,\mathcal K.
\]
Proving Theorem \ref{thm:Hcontainsallfrequencies} reduces to proving $\Lambda_{\mathcal K} = \Z^n$. To this end, we introduce for $L\geq 1$,
\[
\mathfrak C_L := \{ \ell\in\Z^n,\; |\ell|_\infty\leq L\}.
\]
The rest of this section contains an induction proof that $\cup_{L\in\N}\mathfrak C_L\subset\Lambda_{\mathcal K}$. Lemmata \ref{lem:initialization}, \ref{lem:2cube}, \ref{lem:3cube} contain the case $L=1$, $L=2$ and $L=3$ respectively. From there, we perform an induction in Lemma \ref{lem:Induction_step}. 
For a vector $v\in \Z^n,$ we define \[
\supp v :=\{j\in\{1,\dots,n\},\; (v,e_j)\neq 0\}.
\]
Note that, as $\mathsf H_\infty\ni\mathsf Q(\cos(\ell\cdot\theta)) +\mathsf Q(\sin(\ell\cdot\theta)) =\mathsf Q(\ell\cdot\theta), $
\begin{equation}\label{eq:cosandsinimplyell}
    \ell\in \Lambda_{\mathcal H} \quad \Longrightarrow \quad \mathsf Q(\ell\cdot\theta) \in \mathsf H_\infty
\end{equation}

\begin{Lemma}[Initialization]\label{lem:initialization} We have 
\[
\pm\mathsf B(e_j\cdot\theta,e_k\cdot\theta)\in\mathsf H_\infty,\quad j,k\in\{1,\dots,n\}.
\]
    As a consequence,
    \[
    \mathfrak C_1\subset \Lambda_{\mathcal K}, \quad \text{and}\quad\Lambda_{\mathcal H} = \Lambda_{\mathcal K}
    \]
\end{Lemma}
\begin{proof}
Clearly $\mathfrak C_1\subset \Lambda_{\mathcal H}$. For any $\ell\in \mathfrak C_1$, from \eqref{eq:cosandsinimplyell} we have $\mathsf Q( \ell\cdot \theta)\in\mathsf H_\infty.$ Hence, from the fact that $\R\subset \mathsf E$ along with from Corollary \ref{cor:trace-constant},
    \begin{equation}\label{eq:-B_in_H}
    -\mathsf Q( e_j\cdot\theta) = -c^* +\sum_{k\neq j } \mathsf Q( e_k\cdot\theta) \in \mathsf H_\infty.
    \end{equation}
   If $j=k$, then \eqref{eq:-B_in_H} yields $\pm \mathsf B(e_j\cdot\theta,e_j\cdot\theta)=\pm\mathsf Q(e_j\cdot\theta) \in\mathsf H_\infty$. For $j\neq k$ as $e_j\pm e_k\in\mathfrak C_1$, by Proposition \ref{prop:stability_scalar_product}
    \begin{equation}\label{eq:tilde_B_in_H}
    \pm \mathsf B(e_j\cdot\theta,e_k\cdot\theta)\in\mathsf H_\infty,\quad j\neq k.
    \end{equation}
    
 The first part of the result is proved.  For the second part of the statement, we know that $\Lambda_{\mathcal K}\subset\Lambda_{\mathcal H}.$ Let $\ell\in\Lambda_{\mathcal H},$ we write $\ell=\sum_{k=1}^n\ell_ke_k$. From \eqref{eq:cosandsinimplyell},
     \[
    \mathsf Q( \ell\cdot \theta)\in  \mathsf H_\infty. 
    \]
    Moreover, from \eqref{eq:-B_in_H} and \eqref{eq:tilde_B_in_H},
    \[
    -\mathsf Q(\ell\cdot\theta) = -\sum_{j,k=1}^n\ell_j\ell_k \mathsf B (e_j\cdot\theta,e_k\cdot\theta)\in\mathsf H_\infty.
    \]
    Hence,
    \[
    \begin{cases}
    - \mathsf Q( \cos(\ell\cdot\theta)) =  \mathsf Q(\sin(\ell\cdot\theta)) - \mathsf Q( \ell\cdot \theta) \in \mathsf H_\infty\\
   - \mathsf Q(  \sin(\ell\cdot\theta))  =  \mathsf Q( \cos(\ell\cdot\theta)) - \mathsf Q( \ell\cdot \theta) \in \mathsf H_\infty,
\end{cases}
    \]
and consequently $\cos(\ell\cdot\theta),\sin(\ell\cdot\theta)\in \mathcal K_\infty.$
\end{proof}

\subsection{Saturation for trigonometric polynomials: definition of \textit{pivot}--vectors}
Before going further and to simplify the computations below, for $X= \mathcal H, \mathcal K$, we consider
\[
 X_\infty^\C :=  X_\infty \oplus i  X_\infty, 
\]
and we have that the sets $\Lambda_X$ can be identified with
\[
\Lambda_X = \Big\{\ell\in\Z^n,\; e^{i\ell\cdot\theta}\in X^\C_\infty\Big\}.\]
 Moreover, we extend the real bilinear form $\mathsf B$ to $\C$ as a $\C$ bilinear form. Note that by Proposition \ref{prop:stability_scalar_product}, $\mathsf B (\mathcal K^\C_\infty,\mathcal K^\C_\infty)\subset \mathcal H_\infty^\C.$ For $p,q,m\in\Z^n$, we introduce the following notation
\[
\mathsf P( p,q;m) =  \mathsf B (e^{i(p+m)\cdot\theta},e^{i(q-m)\cdot\theta} ) + \mathsf B(e^{i(p-m)\cdot\theta},e^{i(q+m)\cdot\theta} ), 
\]
and we shall use the following polarization formula
\begin{equation}\label{eq:polarization}
\mathsf P( p,q;m) = 2( \mathsf Q(m\cdot\theta)e^{i(p+q)\cdot\theta}+\mathsf B( e^{ip \cdot \theta} , e^{iq\cdot\theta})) .
\end{equation}
In what follows, we say that a vector $p\in\Z^n$ is a $v$--\textit{pivot} vector, $v\in\Z^n$, if 
\begin{equation}\label{eq:pivot-def}
    p,p\pm v\in \Lambda_{\mathcal K}.
\end{equation}
At several places, one shall use the following result.
\begin{Lemma}\label{lem:pivotdansH}
    If $p,q\in\Z^n$ are $v$--\textit{pivot} vectors, $v\in\Z^n$, then
    \[
    \mathsf P( p,q;v) - 2 \mathsf B( e^{ip \cdot \theta} , e^{iq\cdot\theta}) \in\mathcal H^\C_\infty.
    \]
    In particular, if $p,q\in\Z^n$ are $e_j$--\textit{pivot} vectors for any $j\in\{1,\dots,n\}$, then $p+q\in\Lambda_{\mathcal K}.$
\end{Lemma}
\begin{proof}
    The first part follows from a direct application of Proposition \ref{prop:stability_scalar_product}. Moreover, from \eqref{eq:polarization} along with Corollary \ref{cor:trace-constant},
    \begin{equation}\label{eq:using_polarisation_formula}
    \sum_{j=1}^n\Big(\mathsf P( p,q;e_j) - 2 \mathsf B( e^{ip \cdot \theta} , e^{iq\cdot\theta})\Big) = 2e^{i(p+q)\cdot\theta}\sum_{j=1}^n\mathsf Q(e_j\cdot\theta)= 2c^*e^{i(p+q)\cdot\theta}.
    \end{equation}
Hence $p+q\in\Lambda_{\mathcal H}$, and from Lemma \ref{lem:initialization}, $p+q\in\Lambda_{\mathcal K}.$
\end{proof}
As an intermediate step, we shall need the following result.
\begin{Lemma}\label{lem:star-push}
For any $L\in \N^*$,
\[
\mathfrak C_L\subset \Lambda_{\mathcal K} \quad \Longrightarrow \quad (L+1)\mathfrak C_1\subset \Lambda_{\mathcal K}
\]
\end{Lemma}
    \begin{proof}
      We perform an induction on $\Card(\supp \ell)$, for $\ell \in (L+1)\mathfrak C_1$.\\
      
   \begin{itemize}
       \item $\Card(\supp \ell) =1$. Since the cubes $\mathfrak C_L$ are invariant under permutations and sign changes of coordinates, we may assume, without loss of generality, that $\ell=(L+1)e_1$. Thus,
        \begin{equation*}
\sum_{j=2}^n\mathsf P(Le_1,e_1;e_j) - 2\Big(n-1+\frac1L\Big) \mathsf B (e^{ie_1 \cdot\theta},e^{iLe_1 \cdot\theta})
=2e^{i(L+1)e_1\cdot\theta} \sum_{j=1}^n \mathsf Q( e_j\cdot\theta)
=2c^*e^{i(L+1)e_1\cdot\theta}.
\end{equation*}
       from Corollary \ref{cor:trace-constant}. As the left hand side is in $\mathcal H^\C_\infty$ by Proposition \ref{prop:stability_scalar_product}, we deduce that $(L+1)e_j\in \Lambda_{\mathcal K}$, $j\in\{1,\dots,n\},$ by Lemma \ref{lem:initialization}.\\
       
       \item Assume that for some $J\in\{1,\dots,n-1\},$ we have $\{\ell \in (L+1)\mathfrak C_1,\; \Card(\supp \ell)\leq J\}\subset \Lambda_{\mathcal 
       K}$. Let $\ell\in (L+1)\mathfrak C_1$ such that $\Card(\supp \ell)=J+1.$ By symmetry of the problem, we may assume
       \[
       \ell = (L+1)\sum_{j=1}^{J+1}e_j =: (L+1)\ell'.
       \]
      Thus, arguing as above,
  \begin{multline*}
\frac{1}{(L+1)^2}\sum_{j=1}^{J+1} \mathsf B\big(e^{i(L+1)(\ell'-e_j)\cdot\theta},e^{i(L+1)e_j\cdot\theta}\big)
+\frac12\sum_{j=J+2}^n \mathsf P(L\ell',\ell';e_j)\\
-\Big(n-J-1+\frac1L\Big) \mathsf B\big( e^{iL\ell'\cdot\theta},e^{i\ell'\cdot\theta}\big)
= e^{i\ell\cdot\theta}\sum_{j=1}^n \mathsf Q(e_j\cdot\theta) = c^*e^{i\ell\cdot\theta}.
\end{multline*}
        from Corollary \ref{cor:trace-constant}. As the left hand side is in $\mathcal H^\C_\infty$ by Proposition \ref{prop:stability_scalar_product}, we deduce that $\ell\in \Lambda_{\mathcal K}$ by Lemma \ref{lem:initialization}.
        \end{itemize}
    \end{proof}

\subsection{
Saturation for trigonometric polynomials: the
\texorpdfstring{$2$--cube}{2-cube}}
\begin{Lemma}\label{lem:2cube}
    We have 
    \[
    \mathfrak C_2 \subset \Lambda_{\mathcal K}.
    \]
\end{Lemma}
\begin{proof}
     Let $\ell\in\mathfrak C_2$. By symmetry of the cubes, we may assume, without loss of generality, that $\ell\geq 0.$ Note that from Lemma \ref{lem:initialization} along with Lemma \ref{lem:star-push}, we have
   \begin{equation}\label{eq:2C_1_in_Lambda}
   2\mathfrak C_1 \subset \Lambda_{\mathcal K}.
   \end{equation}
   We denote by $R_\ell:=\Card\{ |(\ell,e_j)|=2,\; j\in\{1,\dots,n\} \}$. We separate two cases.\\
\begin{itemize}
    \item Case $R_\ell=1.$ Without loss of generality, we can assume that
    $
    \ell = 2e_1 +\sum_{j=2}^{J} e_j,
    $
    for some $2\leq J\leq n.$ On the one hand,
    \begin{equation}\label{eq:obtain1}
  \mathcal H_\infty^\C\ni   \mathsf B(e^{i 2e_1\cdot\theta},e^{i(\ell-2e_1)\cdot\theta})-2 \mathsf B(e^{i e_1\cdot\theta},e^{i(\ell-e_1)\cdot\theta})  = 2 \mathsf Q(e_1\cdot\theta) e^{i\ell\cdot\theta}.
    \end{equation}
    On the other hand, we consider
    \[
    m_\varepsilon = \sum_{j=2}^{J} \varepsilon_j e_j,\quad \varepsilon\in \{-1,1\}^{J-1}.
    \]
    Remark that, from Lemmata \ref{lem:initialization} and \ref{lem:star-push},
    \begin{equation}\label{eq:m_varepsilon_in_Lambda}
    m_\varepsilon \in \mathfrak C_1,\quad 
    \ell-m_\varepsilon\in 2\mathfrak C_1,
    \end{equation}
    and
    \begin{align*}
  \sum_{\varepsilon\in \{-1,1\}^{J-1} } \mathsf B (e^{i(\ell-m_\varepsilon)\cdot\theta} , e^{i m_{\varepsilon}\cdot\theta} ) & = e^{i\ell\cdot\theta} \sum_{\varepsilon\in \{-1,1\}^{J-1} } (\mathsf Q(m_\varepsilon\cdot \theta) -\mathsf B(\ell\cdot\theta,m_\varepsilon\cdot\theta)).
 \end{align*}
 Yet,
 \begin{align*}
 \sum_{\varepsilon\in \{-1,1\}^{J-1} } \mathsf Q(m_\varepsilon\cdot \theta)   = \sum_{j,k=2}^J\sum_{\varepsilon\in \{-1,1\}^{J-1} } \varepsilon_{j}\varepsilon_{k}\mathsf B(e_j\cdot\theta,e_k\cdot \theta)
  =2^{J-1} \sum_{j=2}^J\mathsf Q(e_j\cdot\theta)
 \end{align*}
 as for $2\leq j\neq k\leq J$ fixed, we have the disjoint union with equal cardinals
 \[
 \{-1,1\}^{J-1} = \{ \varepsilon \in\{-1,1\}^{J-1},\varepsilon_j=\varepsilon_k\} \cup  \{ \varepsilon \in\{-1,1\}^{J-1},\varepsilon_j=-\varepsilon_k\}.
 \]
 Arguing the same way,
 \[
 \sum_{\varepsilon\in \{-1,1\}^{J-1} } \mathsf B(\ell\cdot\theta,m_\varepsilon\cdot\theta) =0.
 \]
With \eqref{eq:m_varepsilon_in_Lambda}, we then have obtained
 \begin{equation}\label{eq:obtainfrom2toJ}
 \mathcal H^\C_\infty\ni 2^{2-J}\sum_{\varepsilon\in \{-1,1\}^{J-1} } \mathsf B (e^{i(\ell-m_\varepsilon)\cdot\theta} , e^{i m_{\varepsilon}\cdot\theta} ) = 2e^{i\ell\cdot\theta}\sum_{j=2}^J\mathsf Q(e_j\cdot\theta).
 \end{equation}
 Finally, considering 
 \[
 p= \sum_{j=1}^Je_j,\quad \ell-p =e_1,
 \]
 we have that $p,\ell-p$, are $e_j$--\textit{pivot} vectors in the sense of \eqref{eq:pivot-def}, for any $j\in\{J+1,\dots,n\}$ and thus
 \begin{equation}\label{eq:fromJ+1ton}
\mathcal H^\C_\infty\ni \sum_{j=J+1}^n\Big(\mathsf P(p,\ell-p;e_j) -2\mathsf B( e^{ip \cdot \theta} , e^{i(\ell-p)\cdot\theta})\Big) = 2e^{i\ell\cdot\theta}\sum_{j=J+1}^n \mathsf Q(e_j\cdot\theta).
 \end{equation}
 Combining \eqref{eq:obtain1},\eqref{eq:obtainfrom2toJ} and \eqref{eq:fromJ+1ton}, we deduce
 \begin{multline*}
 2c^*e^{i\ell\cdot\theta}=2e^{i\ell\cdot\theta}\sum_{j=1}^n\mathsf Q(e_j\cdot\theta) =  \mathsf B(e^{i 2e_1\cdot\theta},e^{i(\ell-2e_1)\cdot\theta})\\
 -2 \mathsf B(e^{i e_1\cdot\theta},e^{i(\ell-e_1)\cdot\theta})
 +2^{2-J}\sum_{\varepsilon\in \{-1,1\}^{J-1} }\mathsf B (e^{i(\ell-m_\varepsilon)\cdot\theta} , e^{i m_{\varepsilon}\cdot\theta} ) \\+\sum_{j=J+1}^n\Big(\mathsf P(p,\ell-p;e_j) -2\mathsf B( e^{ip \cdot \theta} , e^{i(\ell-p)\cdot\theta})\Big)\in\mathcal H_\infty^\C,
 \end{multline*}
 which proves that 
 \begin{equation}\label{eq:R=1IncaseC2}
     \{\ell\in\mathfrak C_2,\;R_\ell=1\}\subset\Lambda_{\mathcal K}.
 \end{equation}
    \item Case $R_\ell>1.$ In such a case, we assume without loss of generality, that $\ell=\sum_{j=1}^n \ell_j e_j$, with $0\leq \ell_j\leq 2$. We set
    \[
    p =\sum_{j=1}^n\lfloor \frac{\ell_j}{2}\rfloor e_j,\quad \ell-p  =\sum_{j=1}^n\lceil \frac{\ell_j}{2}\rceil e_j.
    \]
    Note that $p$ and $\ell-p$ are $e_j$--\textit{pivot} vectors, in the sense of \eqref{eq:pivot-def}, for any $j\in\{1,\dots,n\}$. Indeed, $0\leq \lfloor \tfrac{\ell_j}{2}\rfloor, \lceil \tfrac{\ell_j}{2}\rceil\leq 1$ for $j\in\{1,\dots,n\},$ and thus
    \[
    p\pm e_j, \ell-p\pm e_j \in  \mathfrak C_1\cup \{ m\in\mathfrak C_2,\; R_m=1\}\subset \Lambda_{\mathcal K},
    \]
    from \eqref{eq:R=1IncaseC2}. From Lemma \ref{lem:pivotdansH}, $\ell\in\Lambda_{\mathcal K}$. 
\end{itemize}
\end{proof}
\subsection{
Saturation for trigonometric polynomials: the
\texorpdfstring{$3$--cube}{3-cube}}
\begin{Lemma}\label{lem:3cube}
    We have 
    \[
    \mathfrak C_3 \subset \Lambda_{\mathcal K}.
    \]
    \end{Lemma}
    \begin{proof}
         From the previous step, we have $\mathfrak C_2\subset\Lambda_{\mathcal K}$, and it follows from Lemma \ref{lem:star-push} that $3\mathfrak C_1\subset \Lambda_{\mathcal K}.$
For $\ell \in\mathfrak C_3$, we denote by $R_\ell:=\Card\{ |(\ell,e_j)|=3,\; j\in\{1,\dots,n\} \}.$ Let $\ell\in \mathfrak C_3$. We separate two cases.\\
\begin{itemize}
    \item Case $R_\ell=1$. In such a case, by the symmetry of the cubes, we assume without loss of generality, that $\ell=3e_1+\sum_{j=2}^n \ell_j e_j$, with $0\leq \ell_j\leq 2$ for $j\in\{2,\dots,n\}.$ We set
    \[
    p' = 2e_1,\quad \ell-p' = e_1+\sum_{j=2}^n \ell_j e_j.
    \]
    Note that $p'$ and $\ell-p'$ are $e_1$-\textit{pivot} vectors, in the sense of \eqref{eq:pivot-def}. Indeed,
    \begin{equation*}
    p'-e_1\in \mathfrak C_1\subset\Lambda_{\mathcal K},\quad p'+e_1\in 3\mathfrak C_1\subset\Lambda_{\mathcal K} ,\quad
       \ell-p'\pm e_1\in \mathfrak C_2\subset\Lambda_{\mathcal K}.
    \end{equation*}
    As above, by the polarization formula \eqref{eq:polarization},
    \[
  \mathcal H^\C_\infty\ni \mathsf P( p',\ell-p';e_1) -2\mathsf B(e^{ip'\cdot\theta},e^{i(\ell-p')\cdot\theta}) = 2 \mathsf Q(e_1\cdot\theta)e^{i\ell\cdot\theta}
    \]
    To generate other directions, we use
    \[
    p=\sum_{j=1}^{n}\lfloor \frac{\ell_j}{2} \rfloor e_j,\quad \ell-p =\sum_{j=1}^n\lceil \frac{\ell_j}{2} \rceil e_j.
    \]
    Again, $p$ and $\ell-p$ are $e_j$--\textit{pivot} vectors in the sense of \eqref{eq:pivot-def} for all $j\in\{2,\dots,n\}$. Indeed, as $0\leq \lfloor \tfrac{\ell_j}{2} \rfloor,\lceil \tfrac{\ell_j }{2}\rceil \leq 1$ for $j\in\{2,\dots,n\},$
    \[
    p\pm e_j\in \mathfrak C_2,\quad \ell-p\pm e_j\in \mathfrak C_2,\quad j\in\{2,\dots,n\}.
    \]
     Hence, again by the polarization formula \eqref{eq:polarization},
    \[
  \mathcal H^\C_\infty\ni \sum_{j=2}^n \mathsf P( p,\ell-p;e_j) -2(n-1)\mathsf B(e^{ip\cdot\theta},e^{i(\ell-p)\cdot\theta})= 2\sum_{j=2}^n \mathsf Q(e_j\cdot\theta)e^{i\ell\cdot\theta}.
    \]
    Summing up,
    \begin{multline*}
      \mathcal H^\C_\infty\ni -2(n-1)\mathsf B(e^{ip\cdot\theta},e^{i(\ell-p)\cdot\theta})
     - 2\mathsf B(e^{ip'\cdot\theta},e^{i(\ell-p')\cdot\theta})
      \\+\mathsf P( p',\ell-p';e_1)+\sum_{j=2}^n \mathsf P( p,\ell-p;e_j) 
      = 2e^{i\ell\cdot\theta}\sum_{j=1}^n\mathsf Q(e_j\cdot\theta)  = 2c^*e^{i\ell\cdot\theta}
    \end{multline*}
      from Corollary \ref{cor:trace-constant}. As the left hand side is in $\mathcal H^\C_\infty$ by Proposition \ref{prop:stability_scalar_product}, we deduce that $\ell\in \Lambda_{\mathcal K}$ by Lemma \ref{lem:initialization}. We have proven
      \begin{equation}\label{eq:R=1}
          \{ m\in\mathfrak C_3,\; R_m=1\}\subset \Lambda_{\mathcal K}.
      \end{equation}
      \item Case $R_\ell>1$. In such a case, we assume without loss of generality, that $\ell=\sum_{j=1}^n \ell_j e_j$, with $0\leq \ell_j\leq 3$. We set
    \[
    p =\sum_{j=1}^n\lfloor \frac{\ell_j}{2}\rfloor e_j,\quad \ell-p  =\sum_{j=1}^n\lceil \frac{\ell_j}{2}\rceil e_j.
    \]
    Note that $p$ and $\ell-p$ are $e_j$--\textit{pivot} vectors, in the sense of \eqref{eq:pivot-def}, for any $j\in\{1,\dots,n\}$. Indeed, $0\leq \lfloor \tfrac{\ell_j}{2}\rfloor, \lceil \tfrac{\ell_j}{2}\rceil\leq 2$ for $j\in\{1,\dots,n\},$ and thus
    \[
    p\pm e_j, \ell-p\pm e_j \in  \mathfrak C_2\cup \{ m\in\mathfrak C_3,\; R_m=1\}\subset \Lambda_{\mathcal K},
    \]
    from \eqref{eq:R=1}. Using Lemma \ref{lem:pivotdansH}, we obtain $\ell\in\Lambda_{\mathcal K}$. This ends this step.\\
\end{itemize}
    \end{proof}

\subsection{
Saturation for trigonometric polynomials: the
\texorpdfstring{$L$--cubes, $L\geq 4$}{L-cubes, L >= 4}}

\begin{Lemma}[Induction step]\label{lem:Induction_step}
  We have, for any $L\geq 3$,
    \[
    \mathfrak C_L \subset \Lambda_{\mathcal K}
    \]
\end{Lemma}
\begin{proof}

Assume that $\mathfrak C_L\subset \Lambda_{\mathcal K}$ for some $L\geq 3$. We shall prove that $\mathfrak C_{L+1}\subset \Lambda_{\mathcal K}$. Let $\ell\in\mathfrak C_{L+1}.$ Without loss of generality, by symmetry of the cubes, assume that $\ell\geq 0$, that is all the coefficients of $\ell$ in the canonical basis $(e_j)_j$ are non-negative. Let 
   \[
   p = \lfloor \frac{L+1}{2}\rfloor\sum_{j=1}^n e_j
   .\]
   Note that $p$ and $\ell-p$ are $e_j$--\textit{pivot} vectors in the sense of \eqref{eq:pivot-def}, for any $j\in\{1,\dots,n\}$. Indeed, for any $k\in \{1,\dots,n\},$
   \[
   (p\pm e_j)_k = \begin{cases} \lfloor \tfrac{L+1}{2} \rfloor &\text{if }k\neq j\\ \lfloor \tfrac{L+1}{2} \rfloor\pm 1 &\text{if } k=j \end{cases}
   \]
   In any case,
   \[
     | (p\pm e_j)_k | \leq \lfloor \frac{L+1}{2} \rfloor+1 \leq L
   \]
   if $L\geq 2$. In the same way, for $L\geq 3$
    \[
     | (\ell-(p\pm e_j))_k | \leq L,\quad |(\ell-p)_k|\leq L,
   \]
   which proves the claim. From Lemma \ref{lem:pivotdansH}, $\ell\in\Lambda_{\mathcal K}$. This ends the proof, as initialization is proven in Lemma \ref{lem:3cube}.
\end{proof}

\begin{Remark}
    Inspecting the proof of Lemma \ref{lem:Induction_step} (that is the induction for $L\geq 4$), we observe that we were not able to find a general strategy to treat $L\geq 2$ in a unified way. For instance, consider the $2$--cube $\mathfrak C_2$, in dimension 2 (see Figure \ref{pic:2cube}). To reach $2e_1+e_2$ from $\mathfrak C_1$, we see that we need to construct first $2e_1$ and $2(e_1+e_2)$ (which are in $2\mathfrak C_1$) to obtain the three equations we need to combine, to separate the three unknowns
    \[
    \mathsf B(e_1\cdot\theta,e_2\cdot\theta),\quad \mathsf Q (e_1\cdot\theta),\quad \mathsf Q (e_2\cdot\theta).
    \]
    
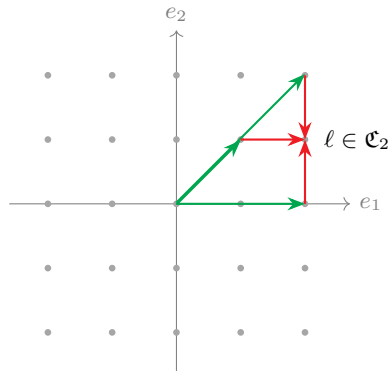
\begin{figure}[ht]
\centering
\begin{tikzpicture}[scale=0.85,
  pt/.style={circle,fill=black!35,inner sep=0pt,minimum size=2.5pt},
  nb/.style={circle,draw=blue!70!black,fill=white,inner sep=0pt,minimum size=4pt}]

  \draw[->,black!50] (-2.6,0) -- (2.7,0) node[right,font=\small] {$e_1$};
  \draw[->,black!50] (0,-2.6) -- (0,2.7) node[above,font=\small] {$e_2$};
  \foreach \x in {-2,...,2} \foreach \y in {-2,...,2} \node[pt] at (\x,\y) {};
  \draw [-Stealth][Green,thick]  (0,0.02) -- (2,2.02);
   \draw [-Stealth][Red,thick]   (2,2) -- (2,1);
 \draw [-Stealth][Green,thick]  (0,-0.01) -- (1,0.99);
   \draw [-Stealth,Red,thick]   (1,1) -- (2,1);
   \draw [-Stealth][Green,thick]  (0,0) -- (2,0);
   \draw [-Stealth][Red,thick]   (2,0) -- (2,1);  
  \node[font=\small,anchor=west] at (2.15,1) {$\ell\in\mathfrak C_2$};
\end{tikzpicture}
\caption{We combine $\mathsf B( e^{i(2e_1+2e_2)\cdot\theta},e^{-ie_2\cdot\theta  }) $, $\mathsf B( e^{i(e_1+e_2)\cdot\theta},e^{ie_1\cdot\theta  }) $ and $\mathsf B( e^{2ie_1\cdot\theta},e^{ie_2\cdot\theta  }), $ to reach $\ell=(2,1)\in\mathfrak C_2$. Note that all the green and red vectors are either in $\mathfrak C_1$ or in $2\mathfrak C_1$.}
\label{pic:2cube}
\end{figure}
    Yet, it appeared that the argument is simpler when the cube $\mathfrak C_L$ is large, that is $L\geq 4$: we can always construct a proper \textit{pivot}--vector, with all combinations staying in the smaller cube $\mathfrak C_{L-1}$ (see for instance an illustration in dimension 2 for $L=5$ in Figure \ref{pic:5cube}). 
\end{Remark}

\begin{figure}[ht]
\centering
\begin{tikzpicture}[scale=0.85,
  pt/.style={circle,fill=black!35,inner sep=0pt,minimum size=2.5pt},
  nb/.style={circle,draw=blue!70!black,fill=white,inner sep=0pt,minimum size=4pt}]

  \draw[->,black!50] (-1.6,0) -- (5.8,0) node[right,font=\small] {$e_1$};
  \draw[->,black!50] (0,-1.6) -- (0,4.8) node[above,font=\small] {$e_2$};
  \foreach \x in {-1,...,5} \foreach \y in {-1,...,4} \node[pt] at (\x,\y) {};
  \node[font=\small,anchor=west] at (5.15,3) {$\ell\in\mathfrak C_5$};
  \draw  (2,2) node[Blue]{$\bullet$};
  \draw  [-,thick][Blue] (2,1)--(2,3);
   \draw  (3,1) node[Blue]{$\bullet$};
   \draw  [-,thick][Blue] (3,0)--(3,2);
    \draw [-Stealth][Green,thick]  (0,0) -- (2,3);
   \draw [-Stealth][Red,thick]   (2,3) -- (5,3);
   \draw [-Stealth][Green,thick]  (0,0) -- (2,1);
   \draw [-Stealth][Red,thick]   (2,1) -- (5,3);
      \draw [-Stealth][Green,thick]  (0,0) -- (3,0);
   \draw [-Stealth][Red,thick]   (3,0) -- (5,3);
    \draw [-Stealth][Green,thick]  (0,0) -- (3,2);
   \draw [-Stealth][Red,thick]   (3,2) -- (5,3);
    \draw [-Stealth][Green,thick]  (0,0) -- (2,2);
   \draw [-Stealth][Red,thick]   (2,2) -- (5,3);
       \draw [-Stealth][Green,thick]  (0,0) -- (3,1);
   \draw [-Stealth][Red,thick]   (3,1) -- (5,3);

\end{tikzpicture}
\caption{To reach $\ell=(5,3)\in\mathfrak C_5$, we use $p=(2,2)$ and $\ell-p=(3,1)$ as \textit{pivot}--vectors. Here, is an illustration of formula \eqref{eq:using_polarisation_formula}, in the $e_2$ direction, to isolate $\mathsf Q(e_2\cdot\theta)e^{i\ell\cdot\theta}$. Remark that all the green and red vectors are in the smaller cube $\mathfrak C_4$.}
\label{pic:5cube}
\end{figure}
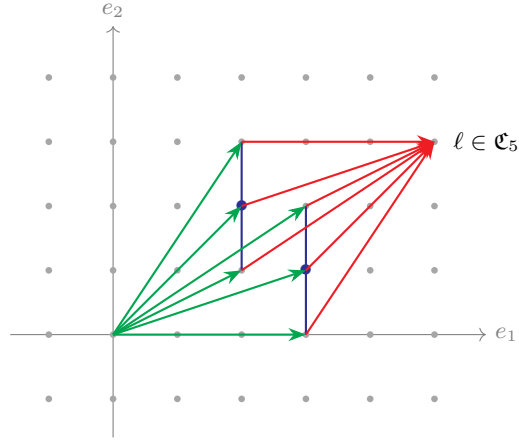
 
\subsection{Proof of the main results}
\begin{proof}[Proof of Theorem \ref{thm:Hcontainsallfrequencies}]
From Lemmata \ref{lem:initialization}, \ref{lem:2cube}, \ref{lem:3cube} and \ref{lem:Induction_step}, we have $\Lambda_\mathcal K = \Z^n$. This is precisely the announced result.
    
\end{proof}

\begin{proof}[Proof of Theorem \ref{thm:main}] From Theorem \ref{thm:Hcontainsallfrequencies}, we observe that the real unital algebra $\mathcal A$ generated by $\{\cos(k\cdot\theta),\sin(k\cdot \theta),\; k \in \mathfrak C_1 \}$ satisfies
\[
\mathcal A \subset \Span\{\cos(\ell\cdot\theta),\sin(\ell\cdot\theta),\;\ell\in\Z^n\}\subset \mathsf H_\infty,
\] 
by trigonometric formulae. Moreover,
\begin{enumerate}
\item As $\mathcal A$ is unital, for any $m\in M$, there exists $f\in\mathcal A$ such that $f(m)\neq 0.$
    \item Let $m_1,m_2\in M$, $m_1\neq m_2$. Since $\Theta$ is injective, there exists $j$ such that $\theta_j(m_1)\neq\theta_j(m_2)$. Moreover, $\theta_j(m_1),\theta_j(m_2)\in[-1,1]\subset(-\pi,\pi)$. Hence either $\cos\theta_j(m_1)\neq\cos\theta_j(m_2)$ or $\sin\theta_j(m_1)\neq\sin\theta_j(m_2)$, and $\cos(e_j\cdot\theta),\sin(e_j\cdot\theta)\in\mathcal A$. Thus $\mathcal A$ separates points.
    \item Let $m\in M$ and $\xi\in T_mM\setminus\{0\}$. Since $\Theta^*\delta=g$, we have $|T_m\theta\xi|_{\R^n}=\|\Theta\|_{L^\infty(M)}^{-1}|\xi|_{g}\neq0$, hence there exists $j$ with $T_m\theta_j\xi\neq0$. Furthermore,
\[
T_m(\cos\circ\theta_j)\xi=-(\sin\circ\theta_j(m))T_m(\theta_j)\xi,\quad
T_m(\sin\circ\theta_j)\xi=\cos\circ\theta_j(m)T_m\theta_j\xi.
\]
As $\sin(\theta_j(m))$, $\cos(\theta_j(m))$ do not vanish simultaneously, we deduce that one of these two quantities is nonzero. Hence, $\mathcal A$ separates tangent vectors.
\end{enumerate}

The Nachbin \cite{Nachbin1949} theorem thus applies, and consequently $\mathcal A\subset\mathsf H_\infty \subset C^\infty(M)$ is dense in $C^\infty(M)$.

\end{proof}

We can finally sketch a proof of our controllability result.
\begin{proof}[Proof of Theorem \ref{th:bilinear_control}]
The proof comes from the techniques developed in \cite{BP25,CP23,DN24,KMP26} and \cite{Par26}. Specifically, \cite[Theorem 6]{CP23} tells us that the saturation property of Theorem~\ref{thm:Hcontainsallfrequencies} implies small-time approximate controllability of phases. Finally, the extra step towards global controllability is given  in \cite[Proposition 1.1]{Par26}, based on the recent result of gradient generation of \cite[Theorem A]{KMP26}.
\end{proof}

\bibliography{biblio_sat}
\bibliographystyle{plain}
\end{document}